\documentclass[11pt, reqno]{amsart}
\usepackage[margin=1in]{geometry}
\usepackage{graphicx}
\usepackage{relsize}
\usepackage{amsmath}
\usepackage{amsthm}
\usepackage{physics}
\usepackage{xcolor}
\usepackage{braket}
\newcommand\numberthis{\addtocounter{equation}{1}\tag{\theequation}}

\numberwithin{equation}{section}
\theoremstyle{plain}
\newtheorem{thm}{Theorem}[section]
\newtheorem*{thm*}{Theorem}

\newtheorem*{cor*}{Corollary}
\theoremstyle{definition}

\theoremstyle{remark}
\newtheorem*{rem}{Remark}

\renewcommand{\c}{\boldsymbol{c}}
\newcommand{\s}{\boldsymbol{s}}

\newcommand{\N}{\mathbb{N}}

\newcommand{\R}{\mathbb{R}}
\newcommand{\C}{\mathbb{C}}

\newcommand{\BigO}[1]{\ensuremath{\operatorname{O}\left(#1\right)}}

\makeatletter
\def\pmod#1{\allowbreak\mkern10mu({\operator@font mod}\,\,#1)} 
\makeatother

\title{On the moments of averages of Ramanujan sums}
\author[Shivani Goel]{Shivani Goel}
\author[M. Ram Murty]{M. Ram Murty}
\address{Indraprashta Institute of Information Technology (IIIT), New Delhi, 110020, India}
\email{shivanig@iiitd.ac.in}
\address{Department of Mathematics and Statistics, Queen's University, Kingston, Canada, ON K7L 3N6.}
\email{murty@queensu.ca}
\begin{document}

\begin{abstract}
Chan and Kumchev studied averages of the first and second moments of Ramanujan sums. In this article, we extend this investigation by estimating the higher moments of averages of Ramanujan sums using the Br\`eteche Tauberian theorem. We also give a result for the moments of averages of Cohen-Ramanujan sums. 
\end{abstract}

\maketitle

\section{\bf Introduction and main results}
For positive integers $q$ and $n$,  Ramanujan \cite{ramanujan1918certain} studied in 1918, the function $c_q(n)$ defined as follows:
\[c_q(n):=\sum_{\substack{1\le j\le n\\(j,q)=1}}e\left(\frac{ nj}{q}\right)=\sum_{\substack{d|n\\d|q}}d\mu\left(\frac{q}{d}\right), \qquad e(t):=e^{2\pi it} .\numberthis\label{rsum}\]
The equality of the two sums is a simple consequence of the M\"obius inversion formula.
Perhaps inspired by the theory of Fourier expansions of continuous functions, Ramanujan first encountered these sums while exploring trigonometric series representations of normalized arithmetic functions of the form $\sum_{q}a_qc_q(n)$, now known as Ramanujan expansions. The sums \eqref{rsum} have since been called Ramanujan sums. Subsequently, in 1932, Carmichael \cite{Carmi} established that these sums also possess an orthogonality property.
   Ramanujan and Carmichael's work set the stage for a general theory of Ramanujan sums and Ramanujan expansions.
   Ramanujan sums exhibit deep connections in number theory and arithmetic, including their role in proving Vinogradov's theorem on the ternary Goldbach problem \cite[Chapter 8]{nathanson1996additive}, Waring-type formulas \cite{konvalina1996generalization}, the distribution of rational numbers in short intervals \cite{jutila2007distribution}, equipartition modulo odd integers \cite{balandraud2007application}, the large sieve inequality \cite{ramare2007eigenvalues}, and various other branches of mathematics. For more recent developments in the direction of Ramanujan expansions, we refer to \cite{HD, AD, LR, schwarz1988ramanujan, SS, AW, EW}. 
   
  Understanding these sums and their distribution is an essential and interesting topic.  Alkan \cite{alkan1, alkan2}  studied the weighted average of Ramanujan sums. The question on the average order over both variables $n$ and $q$ of $c_q(n)$ was first considered by Chan and Kumchev \cite{Chan} motivated by applications to problems on Diophantine approximations of reals by sums of rational numbers. In \cite{Chan}, using both elementary and analytic techniques, they  
found asymptotic formulas for
\begin{equation}\label{moments}
S_{k}(x,y):=\sum_{n\le y}\left(\sum_{q\le x}c_q(n)\right)^k
\end{equation}
for $k=1,2$. Robles and Roy adapted their methodology to compute the averages of generalized Ramanujan sums introduced by Cohen. It is worth noting, however, that in their study presented in \cite{MR3600410}, Robles and Roy claimed a result for higher moments, which has since been determined to be incorrect. To be precise, their Proposition 1.1 implies 
for $k>1$, (and $\beta =1$ in their notation) that
\begin{equation}\label{roy}
S_k(x,y) = {3yx^2 \over \pi^2 } + O(yx\log x + x^{2k}\log ^k x),
\end{equation}
for $y> x^{2k} \log^{k+1} x$.  
This is correct for $k=2$, but for $k=4$, the theorem contradicts itself as can be seen
by a simple application of the Cauchy-Schwarz inequality:
$$ S_2(x,y) \leq y^{1/2} S_4(x,y)^{1/2}.$$
This raises the question of what exactly is the behaviour of (\ref{moments}) for $k\geq 3$. This problem prompts an exploration of the theory of the arithmetical functions of several variables,
a study initiated by Vaidyanathaswamy \cite{vaidya} in 1931.  
This theory is still in evolution and several recent papers
\cite{nogues,  breteche,  sargos} 
highlight the importance of developing such a theory.
In this paper, we derive the asymptotic behaviour of the moments of Ramanujan sums (\ref{moments}) for $k\ge 3$. This result is an important step in developing the theory of the arithmetical functions of several variables. More precisely, we prove:
\begin{thm}\label{maintheorem}
    For $k\ge 3$ and $y> x^k$, as $x\to \infty$, we have 
    \begin{align*}
     S_{k}(x,y)=yx^kQ(\log x)+\BigO{yx^{k-\theta}},
 \end{align*}
 where $Q\in \R[X] $ is a polynomial of exact degree $2^k-2k-1$ and $0\le \theta\le 1.$
\end{thm}
Thus, (\ref{roy}) is egregiously false in two ways: first in the dominant power of $x$ and second, in the powers of the logarithm that must be added to the main term.

 Our main tool is the Br\`eteche Tauberian theorem for
non-negative arithmetical functions of several variables, which we review in the next section.

Analogous questions of moments of these sums over number fields have been studied by numerous mathematicians \cite{CG,chaubey2024moments, MR3332952,nowak12, nowak13,ma2021average,zhai2021average} for the cases $k=1, 2$.  Our methods should extend to handle the number field cases also for $k\geq 3$ and we relegate this to future work.  
\section{ \bf The Br\`eteche Tauberian theorem}
The study of (\ref{moments}) inevitably leads one into the theory of arithmetical functions of several variables.  In the one variable case, the classical Tauberian theorems provide us with asymptotic behaviours of the summatory function of the non-negative arithmetical function of a single variable by relating it with the analytic properties of the associated Dirichlet series.
In the multivariable case, a similar theorem exists but it does not seem to be well-known.  The extension of Cauchy's residue theorem for functions of several variables seems to have been first addressed by Leray 
\cite{leray} in 1959 using the language of sheaf theory.  Later, in the 1980's, Cassou-Nogues \cite{nogues} and Sargos \cite{sargos}  derived more precise results that could be applied to counting problems involving arithmetical functions of several variables.  We should also mention the work of Lichtin 
\cite{lichtin} in this regard.  In the early part of the 21st century, Br\`eteche \cite{breteche} derived a 
multi-variable version of the Tauberian theorem using classical methods of analytic number theory and it is this version that we apply to our situation.
His theorems in this context are as follows.
\begin{thm}\label{tauberian 1}
 Let $f:\mathbb{N}^k\to \R$   be a non-negative function and $F$ the associated Dirichlet series of $f$ defined by 
 \[F(\s)=F(s_1,\cdots,s_k)=\sum_{n_1,\cdots,n_k=1}^{\infty}\frac{f(n_1,\cdots,n_k)}{n_1^{s_1}\cdots n_k^{s_k}}.\]
 Denote by $\mathcal{LR}_k^{+}(\C)$ the set of non-negative $\C$ linear forms from $\C^k$ to $\C$ on $\R_{+}^k$. Moreover, assume that there exists $(c_1,\cdots,c_k)\in \R_{+}^k$ such that:
 \begin{enumerate}
     \item For $\s \in \C^k$, $F(s_1,\cdots,s_k)$ is absolutely convergent for $\Re(s_i)>c_i$  for all $1\le i\le k$.
     \item There exist a finite family $\mathcal{L}=(l^{(i)})_{1\le i\le q}$ of non-zero elements of $\mathcal{LR}_k^{+}(\C)$, a finite family $(h^{(i)})_{1\le i\le q'}$ of elements of $\mathcal{LR}_k^{+}(\C)$ and $\delta_1,\delta_2,\delta_3>0$ such that the function $H$ defined by 
     \[H(\s)=F(\s+\c)\prod_{i=1}^ql^{(i)}(\s)\] has a holomorphic continuation to the domain
     \begin{align*}
        & D(\delta_1,\delta_3)\\&=\left\{\s\in \C^k: \Re{l^{(i)}(\s)}>-\delta_1 \text{ for all}\ i=1,\cdots,q \ \text{and}\  \Re{h^{(i)}(\s)}>-\delta_3 \ \text{for all}\ i=1,\cdots,q' \right\},
     \end{align*}
     and verifies the estimate: for $\epsilon,\epsilon'>0$ we have uniformly in $\s\in  D(\delta_1-\epsilon,\delta_3-\epsilon')$
     \begin{align*}
         H(\s)\ll \prod_{i=1}^q\left(|\Im{l^{(i)}(\s)}|+1\right)^{1-\delta_2\min\left(0,\Re{l^{(i)}(\s)}\right)}(1+(\Im{s_1}+\cdots+\Im{s_k})^{\epsilon}).
     \end{align*}
     Set $J=J(\C)=\{j\in \{1,\cdots, k\}: c_j=0\}$. Denote $w$ to be the cardinality of $J$ and by $j_1<\cdots<j_w$ its elements in increasing order. Define the $w$ linear forms $l^{(q+i)}$  $(1\le i\le w)$ by $l^{(q+i)}(\s)=e^{*}_{j_i}(\s)=s_{j_i}$. 
 \end{enumerate}
 Then, for any $\boldsymbol{\beta}=(\beta_1,\cdots,\beta_k)\in (0,\infty)^k$, there exist a polynomial $Q_{\beta}\in \R[X]$ of degree at most $q+w-\text{Rank}\{l^{(1)},\cdots,l^{(q)}\}$ and $\theta>0$ such that as $x\to \infty$
 \begin{align*}
     \sum_{n_1\le x^{\beta_1}}\cdots \sum_{n_k\le x^{\beta_k}}f(n_1,\cdots,n_k)=x^{<\c,\boldsymbol{\beta}>}Q_{\boldsymbol{\beta}}(\log x)+\BigO{x^{<\c,\boldsymbol{\beta}>-\theta}}.
      \end{align*}
 Here, $<\cdot, \cdot >$ denotes the usual dot product in $\R^k$. 
\end{thm}
The next theorem gives a determination of the precise degree of the polynomial $Q_{\boldsymbol{\beta}}$ appearing in the previous theorem.  Denoting by $\R_*^+$ the set of strictly positive real numbers, the notation 
$\text{con*}(\{l^{(1)},\cdots,l^{(q)}\})$ means 
$\R_*^+ l^{(1)} + \cdots + \R_*^+ l^{(q)}$.
\begin{thm}\label{tauberian 2}
    Let  $f:\mathbb{N}^k\to \R$   be a non-negative function satisfying the assumptions of Theorem \ref{tauberian 1}. Let  $\boldsymbol{\beta}=(\beta_1,\cdots,\beta_k)\in (0,\infty)^k$ and set $\mathcal{B}=\sum_{i=1}^k\beta_ie^{*}_i \in \mathcal{LR}_k^{+}(\C)$. Then, if $\text{Rank}\{l^{(1)},\cdots,l^{(q)}\}=n$, $H(\textbf{0})\ne 0$, and $\mathcal{B}\in \text{con*}(\{l^{(1)},\cdots,l^{(q)}\})$, then $\deg(Q_{\boldsymbol{\beta}})=q+w-n$.
\end{thm}
With these theorems in place, we begin with a review of
multiplicative arithmetical functions of several variables.
\section{\bf A quick review of arithmetical functions of several variables}
The study of arithmetical functions of several variables seems to begin with the fundamental paper of Vaidyanathaswamy \cite{vaidya} written in 1931.  
There, he defines a multiplicative arithmetical function of several variables $f(n_1,..., n_k)$ as a map 
$f:\N^k \to \C$ satisfying the equation
$$f(m_1n_1, ..., m_k n_k) =  f(m_1, ..., m_k)f(n_1, ..., n_k)$$
whenever $(m_1\cdots m_k, n_1\cdots n_k)=1$.
The classical theory of one-variable arithmetical functions extends nicely to the multi-variable case with this definition.  For example, we can define the convolution $f\star g$ of two functions $f$ and $g$ via
$$ (f\star g )(n_1, ..., n_k) = \sum_{d_1|n_1, \, d_2|n_2, \, ..., d_k|n_k} f(d_1, ..., d_k) g(n_1/d_1, ..., n_k/d_k).
$$
If $f$ and $g$ are multiplicative, then it is easy to see that $f\star g$ is also multiplicative.

For multiplicative functions $f$, we can introduce a formal Dirichlet series of several variables along with an Euler product:
$$\sum_{\underline{n}=\underline{1} }^\infty  {f(n_1, ..., n_k)\over n_1^{s_1} \cdots n_k^{s_k}} = \prod_p \left( \sum_{v_1, ..., v_k=0}^\infty {f(p^{v_1}, ..., p^{v_k})
\over p^{v_1s_1} \cdots p^{v_ks_k}}\right).  $$
In our context, the function we will study is
\[f(n_1, ..., n_k) := \sum_{d_1|n_1, \, d_2|n_2, \, ..., d_k|n_k} \mu(n_1/d_1) \cdots \mu(n_k/d_k) g(d_1, ..., d_k) \numberthis\label{fdefn}\]
where 
\[g(n_1, ..., n_k) := {n_1 \cdots n_k \over [n_1, ..., n_k]}, \]
where $[n_1, ..., n_k]$ denotes the least common multiple of $n_1, ..., n_k$.  Since $g$ is multiplicative, we see
that $f$ is multiplicative by our remarks above.
\begin{thm}\label{drichlet series lemma}
    For the function $f(n_1,\cdots,n_k)$, we have 
    \[\sum_{n_1,\cdots,n_k=1}^{\infty}\frac{f(n_1,\cdots,n_k)}{n_1^{s_1}\cdots n_k^{s_k}}=\left( \prod_{\substack{I\subseteq [k]\\ |I|\ge 2}}\zeta(s_{I}-|I|+1) \right)E(s_1,\cdots,s_k), \numberthis \label{main}\]
    where $[k]:=\{1,\cdots,k\}$ and for any subset $I=\{l_1,\cdots,l_r\}$ of $[k]$, we have $s_{I}:=s_{l_1}+\cdots+s_{l_r}$ and 
  $E(s_1, ..., s_k)$ is a Dirichlet series absolutely
  convergent for $\Re(s_i)>1 - 1/k$.
\end{thm}
\begin{proof}
That a factorization of the form (\ref{main}) exists is easily proved as follows.
    We first note that $f(d_1,\cdots,d_k)$ is a  convolution of multiplicative functions. Therefore, from \eqref{fdefn}, we have 
    \begin{align*}
        \sum_{n_1,\cdots,n_k=1}^{\infty}\frac{f(n_1,\cdots,n_k)}{n_1^{s_1}\cdots n_k^{s_k}}&=  \sum_{d_1,\cdots,d_k=1}^{\infty}\frac{g(d_1,\cdots,d_k)}{d_1^{s_1}\cdots d_k^{s_k}}\sum_{e_1,\cdots,e_k=1}^{\infty}\frac{\mu(e_1)\cdots \mu(e_k)}{e_1^{s_1}\cdots e_k^{s_k}}\\
         & =\frac{1}{\zeta(s_1)\cdots\zeta(s_k) }\sum_{d_1,\cdots,d_k=1}^{\infty}\frac{g(d_1,\cdots,d_k)}{d_1^{s_1}\cdots d_k^{s_k}}\numberthis\label{seriesf}
    \end{align*}
    since 
    $${1\over \zeta(s)} = \sum_{n=1}^\infty {\mu(n) \over n^s}. $$
    We examine the series on the right of (\ref{seriesf}) as follows.
    \begin{align*}
        \sum_{d_1,\cdots,d_k=1}^{\infty}\frac{g(d_1,\cdots,d_k)}{d_1^{s_1}\cdots d_k^{s_k}}&=\prod_{p}\left(\sum_{v_1,\cdots,v_k=0}^{\infty}\frac{p^{v_1+\cdots+v_k-\max(v_1,\cdots,v_k)}}{p^{s_1v_1+\cdots+s_kv_k}}\right)\\
        & \prod_{p}\left(\sum_{n=0}^{\infty}p^{-n}\sum_{\substack{v_1,\cdots,v_k=0\\\max(v_1,\cdots,v_k)=n}}^{\infty}\frac{p^{v_1+\cdots+v_k}}{p^{s_1v_1+\cdots+s_kv_k}}\right).\numberthis\label{sumequation}
    \end{align*}
    The Euler factor can be written as
    \[1 + 
    p^{-1}\sum_{\substack{v_1,\cdots,v_k=0\\\max(v_1,\cdots,v_k)=1}}^{\infty}\frac{p^{v_1+\cdots+v_k}}{p^{s_1v_1+\cdots+s_kv_k}} + 
    \sum_{n=2}^{\infty}p^{-n}\sum_{\substack{v_1,\cdots,v_k=0\\\max(v_1,\cdots,v_k)=n}}^{\infty}\frac{p^{v_1+\cdots+v_k}}{p^{s_1v_1+\cdots+s_kv_k}} . \numberthis \label{two}\]
    The inner sum in the second summation is actually
    a finite sum with at most $(n+1)^k$ terms and 
    with $\sigma = \Re(s_i)$, it is easily estimated to be 
    $$ \ll (n+1)^k p^{kn(1-\sigma)}.$$
    This means that 
    $$ \sum_p  \sum_{n=2}^{\infty}p^{-n}\sum_{\substack{v_1,\cdots,v_k=0\\\max(v_1,\cdots,v_k)=n}}^{\infty}\frac{p^{v_1+\cdots+v_k}}{p^{s_1v_1+\cdots+s_kv_k}}  $$
    converges absolutely for $\Re(s_i)> 1-1/k$.  
    We can therefore factor the Euler product in 
     (\ref{sumequation}) to get 
    $$ \sum_{d_1,\cdots,d_k=1}^{\infty}\frac{g(d_1,\cdots,d_k)}{d_1^{s_1}\cdots d_k^{s_k}} =
         \left( \prod_{p}\left( 1 + 
    p^{-1}\sum_{\substack{v_1,\cdots,v_k=0\\\max(v_1,\cdots,v_k)=1}}^{\infty}\frac{p^{v_1+\cdots+v_k}}{p^{s_1v_1+\cdots+s_kv_k}}\right)\right) E^*(s_1, ..., s_k), $$
    where $E^*(s_1, ..., s_k)$ is a Dirichlet series absolutely convergent in $\Re(s_i)>1 - {1\over k}$.  
    The Euler product above can be analyzed as follows.  The $p$-Euler factor can be written as
    $$ 1+ {1\over p}  \sum_{\emptyset\ne I\subseteq [k]}
     \prod_{i\in I} pT_i^{v_i} 
    $$
    where $T_i=p^{-s_i}$.  This observation allows us
    to further factor the Euler product as
    $$ \left( \prod_{\emptyset\ne I\subseteq [k]}\zeta(s_I - |I|+1) \right) E^{**}(s_1, ..., s_k) $$
    where $E^{**}(s_1, ..., s_k)$ is a Dirichlet
    series absolutely convergent for $\Re(s_i)>1/2$. 
    Combining all these observations and setting 
    $$E(s_1, ..., s_k) = E^*(s_1, ..., s_k)E^{**}(s_1, ..., s_k) ,$$
    we obtain
    $$ \sum_{d_1,\cdots,d_k=1}^{\infty}\frac{g(d_1,\cdots,d_k)}{d_1^{s_1}\cdots d_k^{s_k}}
    = \left( \prod_{\emptyset\ne I\subseteq [k]}\zeta(s_I - |I|+1) \right) E(s_1, ..., s_k) .
    $$
    Taking into account   (\ref{seriesf})
    and noting that the singleton sets are removed 
    from our product of zeta functions, we obtain
    (\ref{main}), as claimed. 
    \end{proof}

    \begin{rem}
        Though it is not needed for our purposes, We can 
        determine $E(s_1, ..., s_k)$ very explicitly:
      \begin{align*}
     E(s_1,\cdots,s_k)&
     =\prod_p\frac{\sum_{\emptyset\ne I\subseteq [k]}(-1)^{|I|}(p^{|I|-s_I}-1)\prod_{\substack{\emptyset\ne J\subseteq [k]\\I\ne J}}(1-p^{|J|-s_J-1})}{1+\sum_{\emptyset\ne I\subseteq [k]}(-1)^{|I|}p^{|I|-s_I}}
 \end{align*}  
   To see this, let 
     $p^{-s_1}=T_1, \cdots, p^{-s_k}=T_k$
     as before.  Then, 
    \begin{align*}
\sum_{\substack{v_1,\cdots,v_k=0\\\max(v_1,\cdots,v_k)=n}}^{\infty}\frac{p^{v_1+\cdots+v_k}}{p^{s_1v_1+\cdots+s_kv_k}}&=\sum_{\substack{v_1,\cdots,v_k=0\\\max(v_1,\cdots,v_k)\le n}}^{\infty}{p^{v_1+\cdots+v_k}}{T_1^{v_1}\cdots T_k^{v_k}}-\sum_{\substack{v_1,\cdots,v_k=0\\\max(v_1,\cdots,v_k)\le n-1}}^{\infty}{p^{v_1+\cdots+v_k}}{T_1^{v_1}\cdots T_k^{v_k}}\\
&=\prod_{i=1}^k\left(\sum_{v_i\le n}p^{v_i}T_i^{v_i}\right)-\prod_{i=1}^k\left(\sum_{v_i\le n-1}p^{v_i}T_i^{v_i}\right)\\
&=\prod_{i=1}^k\left(\frac{1-(pT_i)^{n+1}}{1-pT_i}\right)-\prod_{i=1}^k\left(\frac{1-(pT_i)^{n}}{1-pT_i}\right).\numberthis\label{sumequation in T}
    \end{align*}
    Now, multiplying this by $p^{-n}$ and summing from $n=0$ to $\infty$ gives, (using the abbreviation 
    $T_I=p^{-s_I}$), 
    $$
    \frac{1}{\prod_{i=1}^k(1-pT_i)}\sum_{n=0}^{\infty}\sum_{\emptyset\ne I\subseteq [k]}(-1)^{|I|}\left(p^{n(|I|-1)+|I|}T_I^{n+1}-p^{n(|I|-1)}T_I^{n}\right)$$
    $$
     = \frac{1}{\prod_{i=1}^k(1-pT_i)}\sum_{\emptyset\ne I\subset [k]}(-1)^{|I|}\left(\frac{p^{|I|}T_I}{1-p^{|I|-1}T_I}-\frac{1}{1-p^{|I|-1}T_I}\right)$$
   $$= \frac{1}{\prod_{i=1}^k(1-pT_i)}\sum_{\emptyset\ne I\subseteq [k]}(-1)^{|I|}\left(\frac{p^{|I|}T_I -1}{1-p^{|I|-1}T_I}\right)$$
    
   We therefore have from \eqref{sumequation} and \eqref{sumequation in T}, and the above calculation that
 \begin{align*}
     \sum_{d_1,\cdots,d_k=1}^{\infty}\frac{g(d_1,\cdots,d_k)}{d_1^{s_1}\cdots d_k^{s_k}}
   &= \prod_{p}\frac{1}{\prod_{i=1}^k(1-pT_i)}\sum_{\emptyset\ne I\subseteq [k]}(-1)^{|I|}\left(\frac{p^{|I|}T_I -1}{1-p^{|I|-1}T_I}\right)\\  
     &=\left( \prod_{\emptyset\ne I\subseteq [k]}\zeta(s_{I}-|I|+1)\right) E(s_1,\cdots,s_k).\numberthis\label{gseries}
 \end{align*}
 Here, 
 \begin{align*}
     E(s_1,\cdots,s_k)&=\prod_p\frac{\sum_{\emptyset\ne I\subset [k]}(-1)^{|I|}(p^{|I|}T_I-1)\prod_{\substack{\emptyset\ne J\subset [k]\\I\ne J}}(1-p^{|J|-1}T_J)}{\prod_{i=1}^k(1-pT_i)}\\&
     =\prod_p\frac{\sum_{\emptyset\ne I\subseteq [k]}(-1)^{|I|}(p^{|I|}T_I-1)\prod_{\substack{\emptyset\ne J\subseteq [k]\\I\ne J}}(1-p^{|J|-1}T_J)}{1+\sum_{\emptyset\ne I\subseteq [k]}(-1)^{|I|}p^{|I|}T_{I}}.
 \end{align*}
 From this explicit expression, the region of absolute 
 convergence for $E(s_1, ..., s_k)$ 
 is not immediately clear and thus, we have opted for the more expedient method in the proof of our theorem.
  \end{rem}
\section{\bf Non-negativity of $f(n_1, ..., n_k) $}
In this section, we will show that $f(n_1, ..., n_k)$
is a non-negative function, thus paving the way for
an application of the Br\`eteche Tauberian theorem.
\begin{thm}\label{non negativity theoerem}
$f(n_1, ..., n_k) \geq 0$ for all $(n_1, ..., n_k)\in \N^k$.
\end{thm}
\begin{proof}  Since $f$ is a convolution of two multiplicative functions, it is also multiplicative.  Therefore, to prove 
the lemma, it suffices to show for each prime $p$, 
$$f(p^{v_1}, p^{v_2}, ..., p^{v_k}) \geq 0, \qquad v_1, ..., v_k \geq 0.  $$
Since $f$ is symmetric, we can suppose without any loss of generality that $v_1\geq v_2 \geq  \cdots \geq v_k$.
We proceed by induction on $k$.  For $k=1$, the result is clear.  
We may also suppose that all $v_i\geq 1$ for otherwise, we are again done by induction.
If $v_1>v_2$, then noting that
\begin{equation}\label{key}
\qquad f(p^{v_1}, p^{v_2}, ..., p^{v_k})=
\end{equation}
$$
 \sum_{d_2|p^{v_2},..., d_k| p^{v_k}} \mu(d_2) \cdots \mu(d_k) \left\{ g\left( p^{v_1}, {p^{v_2}\over d_2}, ..., {p^{v_k}\over d_k}
\right) - g\left(p^{v_1-1}, {p^{v_2}\over d_2}, ..., 
{p^{v_k}\over d_k}\right) \right\},$$
we have 
$$g\left( p^{v_1}, {p^{v_2}\over d_2}, ..., {p^{v_k}\over d_k}
\right) = g\left( 1, {p^{v_2}\over d_2}, ..., {p^{v_k}\over d_k}
\right) \left( p^{v_1} , \Big[ {p^{v_2}\over d_2}, ..., {p^{v_k}\over d_k}\Big] \right) $$
and
$$g\left( p^{v_1-1}, {p^{v_2}\over d_2}, ..., {p^{v_k}\over d_k}
\right) = g\left( 1, {p^{v_2}\over d_2}, ..., {p^{v_k}\over d_k}
\right) \left( p^{v_1-1} , \Big[ {p^{v_2}\over d_2}, ..., {p^{v_k}\over d_k}\Big] \right).$$
We see in this case that the gcd in both cases is the same and so the term in braces in (\ref{key}) above is zero.
Now suppose that $v_1=v_2 =\cdots v_{\ell} > v_{\ell+1} \geq \cdots \geq v_k$.
We have
$$f(p^{v_1}, ..., p^{v_k}) = \sum_{d_1|p, ..., d_k|p} \mu(d_1) \cdots \mu(d_k) g(p^{v_1}/d_1, ..., p^{v_k}/d_k). $$
Noting that in the sum over divisors that each $d_i$ can only be $1$ or $p$, we arrange the sum as follows.
We write $d_i=p^{e_i}$  where $e_i=0$ or $1$.  
We can then identify each tuple $(d_1, ..., d_k)$ with a subset $I\subseteq [k]$ where $I=\{ i: e_i=1\}$.  
Our sum becomes
$$f(p^{v_1}, ..., p^{v_k}) = \sum_{I\subseteq [k]} (-1)^{|I|} p^{s - |I| - \max(v_i - e_i: 1\leq i\leq k)}, $$
where 
$$s= v_1 + \cdots v_k. $$
We let $I_0=\{ 1, 2, ..., \ell\}$ and set $s'= v_2 + \cdots + v_k. $  We split the sum on the right into three parts:
$$\sum_{I: I \cap I_0=\emptyset} \qquad + \qquad \sum_{I: \emptyset \neq I \cap I_0 \neq I_0} \qquad + \qquad 
\sum_{I: I\supseteq I_0} . $$
Letting $J= \{ \ell+1 , \cdots , k\}$, the first part is equal to
$$ \sum_{I\subseteq J} (-1)^{|I|} p^{s' - |I|} = p^{s'} \left( 1 - {1\over p}\right)^{k-\ell},$$
because in this case  $\max(v_i-e_i: 1\leq i \leq k) = v_1$.  
In the second part, we again have $\max(v_i-e_i: 1\leq i \leq k) = v_1$ so that the second part is equal to
$$ \sum_{I: \emptyset \neq I \cap I_0 \neq I_0} (-1)^{|I|} p^{s'-|I|} =p^{s'}  \sum_{j=1}^{\ell - 1} {\ell \choose j} (-1)^j p^{-j}
\left( 1 - {1\over p}\right)^{k-\ell}.$$
Finally, in the third part,  $\max(v_i-e_i: 1\leq i\leq k)= v_1-1$ so that the third part equals
$$\sum_{I: I\supseteq I_0} (-1)^{|I|} p^{s' +1 - |I|} = p^{s'+1-\ell} (-1)^\ell \left( 1 - {1\over p}\right)^{k-\ell}. $$
Notice that the first part and the second part combine to give
$$p^{s'} \sum_{j=0}^{\ell - 1} {\ell \choose j} (-1)^j p^{-j}
\left( 1 - {1\over p}\right)^{k-\ell},$$
since the term corresponding to $j=0$ is the contribution from the first part.
Putting everything together gives
$$ p^{s'} \sum_{j=0}^{\ell } {\ell \choose j} (-1)^j p^{-j}
\left( 1 - {1\over p}\right)^{k-\ell}    - (-1)^\ell p^{s' - \ell} \left( 1 - {1\over p}\right)^{k-\ell} + p^{s'+1-\ell} (-1)^\ell \left( 1 - {1\over p}\right)^{k-\ell}. $$
This simplifies to
$$ p^{s'}\left( 1 - {1\over p}\right)^k + (-1)^{\ell} \left( 1 - {1\over p}\right)^{k-\ell} \left( p^{s'+1-\ell} - p^{s'-\ell}\right). $$
Noting that $\ell\geq 2$, we see that this is certainly positive if $\ell$ is even.  
If $\ell$ is odd, the term is equal to
$$ p^{s'}\left( 1 - {1\over p}\right)^k - \left( 1 - {1\over p}\right)^{k-\ell} \left( p^{s'+1-\ell} - p^{s'-\ell}\right). $$
We easily see that this reduces to checking that
$$ p^{s'} (p-1)^\ell \geq p^\ell  \left( p^{s'+1-\ell} - p^{s'-\ell}\right) = p^{s'}(p-1),$$
which is evidently true.  This completes the proof of non-negativity.
\end{proof}
\section{\bf Average order of $f(n_1, ..., n_k) $}
The average order of the function $f(n_1, ..., n_k)$ is the most essential step in the proof of Theorem \ref{maintheorem}. In this section, we estimate an average of  $f(n_1, ..., n_k)$ as an application of Theorems \ref{tauberian 1} and \ref{tauberian 2}.
\begin{thm}\label{average theorem}
 For $0<\theta<1$, we have as $x\to \infty$
 \begin{align*}
     \sum_{n_1,\cdots,n_k\le x}f(n_1, ..., n_k)=x^kQ(\log x)+\BigO{x^{k-\theta}},
 \end{align*}
 where $Q\in \R[X]$ is a polynomial of exact degree $2^k-2k-1.$
\end{thm}
\begin{proof}
     In Theorem \ref{non negativity theoerem}, we proved that  $f(n_1, ..., n_k) $ is non-negative and in Theorem \ref{drichlet series lemma}, we proved $f(n_1, ..., n_k)$ has an absolutely convergent series $F(\s)$ for $\Re(s_i)>1$ for all $1\le i\le k.$ This shows $f(n_1, ..., n_k)$ satisfies  $(1)$ of Theorem \ref{tauberian 1}. Next, we show that $f(n_1, ..., n_k)$ also satisfies  $(2)$ of Theorem \ref{tauberian 1}.  Write $\textbf{1}=(1,\cdots,1)$ then,  $F(\s+\textbf{1})$ is an absolutely convergent series for $\Re(s_i)>0$. Therefore, for the linear forms 
    $\prod_{\substack{I\subseteq [k]\\ |I|\ge 2}}s_{I}$,  define the function 
    \[H(\s):=F(\s+\textbf{1})\prod_{\substack{I\subseteq [k]\\ |I|\ge 2}}s_{I}.\]
    Since $c_i=1$ for all $1\le i\le k$, we take $q'=0$
   in the notation of Theorem \ref{tauberian 1}.  Furthermore, for any $\zeta(s_1+\cdots +s_{\ell}+1)$ in $F(\s+\textbf{1})$, there is a  linear form $(s_1+\cdots +s_{\ell})$ such that $$(s_1+\cdots +s_{\ell})\zeta(s_1+\cdots +s_{\ell}+1)$$ has analytic continuation on the plane $\Re{(s_1+\cdots +s_{\ell})}>-\epsilon$, where $\epsilon>0$ and for $I \subseteq K$, we have $|I|={\ell}\ge 2$. Therefore, $H(\s)$ also has analytic continuation on the plane $\Re{(s_1+\cdots +s_{\ell})}>-\epsilon$. Consider $h^{i}(\s)=s_i$, set $\delta_1=\delta_3=\epsilon$. Moreover, from Lemma \ref{drichlet series lemma}, $E(\s+\textbf{1})$   has analytic continuation on the plane $\Re{(s_1+\cdots +s_{\ell})}>-\epsilon$. We know that for $\Re{s_i}>-1$ and for all $\epsilon_0>0$
    \[s_I\zeta(1+s_I)\ll_{\epsilon_0}(1+|s_I|)^{1-\frac{1}{2}\min(0,\Re{s_I})+\epsilon_0}.\] The above argument shows that $H(\s)$ satisfies $(2)$ of Theorem \ref{tauberian 1} with $\delta_2=1/2$. Therefore, we have as $x\to \infty, $ 
 \begin{align*}
     \sum_{n_1,\cdots,n_k\le x}f(n_1, ..., n_k)=x^kQ(\log x)+\BigO{x^{k-\theta}},
 \end{align*}
 where $Q(\log x)$ is a polynomial of degree at most $2^k-2k-1$.
      
       Next, $c_i>0$ for all $1\le i\le k$, this implies $w=0.$ Again, it is easy to see that the rank of the collection of linear forms $s_{I}$ is $k$ and the interior of the cone generated by linear forms is the set $\mathcal{B}=\sum_{i=1}^k\beta_i\textbf{e}_i^*$ for $\boldsymbol{\beta}=(\beta_1,\cdots,\beta_k)\in(0,\infty)^k$ 
    and $\textbf{e}_i^*(\s)=s_i$. Also, as $s_{I}\to 0$, $s_{I}\zeta(s_I+1)\to 1$ and hence $H(\textbf{0})\ne 0$ Thus, from Theorem \ref{tauberian 2}, $\deg(Q)=2^k-2k-1.$ This gives the required result.
\end{proof}

\section{\bf Higher Moments of Ramanujan sums}
In this section, we provide a proof of Theorem \ref{maintheorem} using the average order of $f(n_1, ..., n_k)$ obtained in the previous section.
\begin{proof}[Proof of Theorem \ref{maintheorem}]
   From the definition of Ramanujan sums, we have
\begin{align*}
    S_k(x,y)&=\sum_{n\le y}\left(\sum_{q\le x}c_q(n)\right)^k=\sum_{n\le y}\sum_{q_1,\cdots,q_k\le x}\sum_{\substack{d_1|q_1\\d_1|n}}d_1\mu\left(\frac{q_1}{d_1}\right)\cdots\sum_{\substack{d_k|q_k\\d_k|n}}d_k\mu\left(\frac{q_k}{d_k}\right)\\&
    =\sum_{q_1,\cdots,q_k\le x}\sum_{\substack{d_1|q_1, \cdots,d_k|q_k}}d_1\cdots d_k\mu\left(\frac{q_1}{d_1}\right)\cdots \mu\left(\frac{q_k}{d_k}\right)\sum_{\substack{n\le y\\ [d_1,\cdots,d_k]|n}}1 \\&
    =\sum_{q_1,\cdots,q_k\le x}\sum_{\substack{d_1|q_1, \cdots, d_k|q_k}}d_1\cdots d_k\mu\left(\frac{q_1}{d_1}\right)\cdots \mu\left(\frac{q_k}{d_k}\right)\left({\frac{y}{[d_1,\cdots,d_k]}}+\BigO{1}\right)\\&
    =y\sum_{q_1,\cdots,q_k\le x}\sum_{\substack{d_1|q_1, \cdots, d_k|q_k}}\frac{d_1\cdots d_k}{[d_1,\cdots,d_k]}\mu\left(\frac{q_1}{d_1}\right)\cdots \mu\left(\frac{q_k}{d_k}\right)+\BigO{x^{2k}}.
\end{align*}
The inner sum is precisely $f(q_1, ..., q_k)$ in
our notation by virtue of (\ref{fdefn}).     Thus, 
from  Theorem \ref{average theorem}, for $0<\theta<1$, we have as $x\to \infty$
 \begin{align*}
     S_{k}(x,y)=yx^kQ(\log x)+\BigO{yx^{k-\theta}+x^{2k}},
 \end{align*}
 where $Q(\log x)$ is a polynomial of degree $2^k-2k-1.$  This completes the proof.
\end{proof}

\section{\bf Moments of Cohen Ramanujan sums}
In \cite{cohen1949extension}, Cohen generalized the Ramanujan sums in the following way:
\[c_q^{\beta}(n) :=\sum_{\substack{1\le j\le q^{\beta}\\(j,q^{\beta})_{\beta}=1}}e\left(\frac{jn}{q^{\beta}}\right)=\sum_{\substack{d|q\\d^{\beta}|n}}d^{\beta}\mu\left(\frac{q}{d}\right).\]We refer to these  sums as Cohen-Ramanujan sums. Here, $(m,n)_{\beta}$ denotes the generalized gcd  which is the largest $l^{\beta}$ dividing both $m$ and $n.$ In \cite{MR3600410}, Robles and Roy estimated the averages of moments of Cohen-Ramanujan sums but their result is correct only for the first and second moments. 
For higher moments, it is false.
Our analysis of the previous sections is amenable to treat this general case too.  
Therefore, applying  our method to obtain the higher moments of these sums, we get:
\begin{thm}\label{theorem2}
    For $k\ge 3$ and $y> x^k$, as $x\to \infty$, we have 
    \begin{align*}
   \sum_{n\le y}\left(\sum_{q\le x}  c_q^{\beta}(n)\right)^k=yx^{k(\beta+1)/2}Q(\log x)+\BigO{yx^{k(\beta+1)/2-\theta}},
 \end{align*}
 where $Q(\log x)$ is a polynomial of exact degree $2^k-2k-1$ and $0\le \theta\le 1.$
\end{thm}
One can prove Theorem \ref{theorem2} by slightly modifying the proof of Theorem \ref{maintheorem}, taking into account the (minor) differences in the definitions between Ramanujan sums and the Cohen-Ramanujan sums.  

\section{\bf Concluding remarks}
There is undoubtedly a deeper significance of our main theorem to our current understanding of the Riemann hypothesis.  Indeed, Ramanujan (see formula (7.2) in \cite{ramanujan1918certain}) showed that 
$${\sigma_{1-s} (n) \over \zeta(s)} = \sum_{q=1}^\infty {c_q(n) \over q^s} ,\qquad \sigma_w(n):= \sum_{d|n} d^w , $$
valid for $\Re(s) >1$.  As the left hand side admits
a meromorphic continuation to the entire complex plane with poles located at the zeros of the zeta function, we can see from standard analytic number theory that the Riemann hypothesis is equivalent to
the estimate
$$\sum_{q\leq x} c_q(n) = O(x^{{1\over 2} + \epsilon}) ,$$
for any $\epsilon >0$ and for any fixed $n$.
The implied constant in the estimate depends on $n$
and a careful analysis gives $O((xn)^{{1\over 2} + \epsilon} )$ as the final estimate for the sum.
We therefore can expect $O(y^{1+ k/2 +\epsilon} x^{k/2+\epsilon} )$ as a final estimate for $S_k(x,y)$ assuming the Riemann hypothesis.   
What our result shows is that this is unconditionally true if 
$y$ is around $x$.  

$$\quad $$
\noindent Acknowledgements.  We thank Sneha Chaubey
for comments on an earlier draft of this paper.
 \bibliographystyle{plain}
   \bibliography{ref}
\end{document}